\documentclass{amsart}

\usepackage{amsmath,amsthm,mathrsfs,amssymb,graphicx}
\usepackage[all]{xy}
\usepackage{tikz}
\usetikzlibrary{shapes,snakes}
\usepackage[unicode=true,
					 pdfusetitle=true,
					 bookmarks=true,
					 bookmarksnumbered=false,
					 bookmarksopen=false,
					 breaklinks=true,pdfborder={0 0 0},
					 backref=false,
					 colorlinks=true,
					 plainpages=false,
					 pdfpagelabels,
					 linktocpage=true]{hyperref}

\newtheorem{theorem}{Theorem}[section]
\newtheorem{corollary}[theorem]{Corollary}

\newtheorem{lemma}[theorem]{Lemma}
\newtheorem{proposition}[theorem]{Proposition}

\theoremstyle{definition}
\newtheorem{definition}[theorem]{Definition}
\newtheorem{question}[theorem]{Question}
\newtheorem{remark}[theorem]{Remark}

\newtheoremstyle{principle}{}{}{\itshape}{}{\bfseries}{.}{.5em}{\thmnote{#3}#1}
\theoremstyle{principle}

\newtheoremstyle{case}{}{}{}{}{\itshape}{.}{.5em}{\thmnote{Case #3}#1}
\theoremstyle{case}

\newcommand{\ran}{\operatorname{ran}}

\newcommand{\pair}[1]{\langle #1 \rangle}

\newcommand{\Iff}{\Longleftrightarrow}
\newcommand{\from}{\Longleftarrow}

\newcommand{\forces}{\Vdash}

\newcommand{\0}{\mathbf{0}}

\begin{document}

\title{On Mathias generic sets}

\author{Peter A.\ Cholak}
\address{Department of Mathematics\\
University of Notre Dame\\
Notre Dame, Indiana 46556 U.S.A.}
\email{cholak@nd.edu}

\author{Damir D.\ Dzhafarov}
\address{Department of Mathematics\\
University of Notre Dame\\
Notre Dame, Indiana 46556 U.S.A.}
\email{ddzhafar@nd.edu}

\author{Jeffry L.\ Hirst}
\address{Department of Mathematical Sciences\\
Appalachian State University\\
\\Boone, North Carolina 28608 U.S.A.}
\email{jlh@math.appstate.edu}

\thanks{The authors are grateful to Christopher P.\ Porter for helpful comments. The first author was partially supported by NSF grant DMS-0800198. The second author was partially supported by an NSF Postdoctoral Fellowship. The third author was partially supported by grant ID\#20800 from the John Templeton Foundation. The opinions expressed in this publication are those of the authors and do not necessarily reflect the views of the John Templeton Foundation.}

\maketitle

\begin{abstract}
We present some results about generics for computable Mathias forcing. The $n$-generics and weak $n$-generics in this setting form a strict hierarchy as in the case of Cohen forcing. We analyze the complexity of the Mathias forcing relation, and show that if $G$ is any $n$-generic with $n \geq 3$ then it satisfies the jump property $G^{(n-1)} = G' \oplus \emptyset^{(n)}$. We prove that every such $G$ has generalized high degree, and so cannot have even Cohen 1-generic degree. On the other hand, we show that $G$, together with any bi-immune set $A \leq_T \emptyset^{(n-1)}$, computes a Cohen $n$-generic set.
\end{abstract}

\section{Introduction}

Forcing has been a central technique in computability theory since it was introduced (in the form we now call Cohen forcing) by Kleene and Post to exhibit a degree strictly between $\0$ and $\0'$. The study of the algorithmic properties of Cohen generic sets, and of the structure of their degrees, has long been a rich source of problems and results. In the present paper, we propose to undertake a similar investigation of generic sets for (computable) Mathias forcing, and present some of our initial results in this direction.

Mathias forcing was perhaps first used in computability theory by Soare in \cite{Soare-1969} to build a set with no subset of strictly higher degree. Subsequently, it became a prominent tool for constructing infinite homogeneous sets for computable colorings of pairs of integers, as in Seetapun and Slaman \cite{SS-1995}, Cholak, Jockusch, and Slaman \cite{CJS-2001}, and Dzhafarov and Jockusch \cite{DJ-2009}. It has also found applications in algorithmic randomness, in Binns, Kjos-Hanssen, Lerman, and Solomon \cite{BKLS-2006}.

We show below that a number of results for Cohen generics hold also for Mathias generics, and that a number of others do not. The main point of distinction is that neither the set of conditions, nor the forcing relation is computable, so many usual techniques do not carry over. We begin with background in Section \ref{sec_defns}, and present some preliminary results in Section \ref{sec_basic}. In Section \ref{sec_forcing} we characterize the complexity of the forcing relation, and in Section 5 we prove a number of results about the degrees of Mathias generic sets, and about their relationship to Cohen generic degrees. We indicate questions along the way we hope will be addressed in future work.

\section{Definitions}\label{sec_defns}

We assume familiarity with the terminology particular to Cohen forcing. (For background on computability theory, see \cite{Soare-TA}. For background on Cohen generic sets, see Section 1.24 of \cite{DH-2010}.) The purpose of this section is to set down the precise analogues for Mathias forcing. These are standard, but their formalizations in the setting of computability theory require some care. A slightly different presentation is given in \cite[Section 6]{BKLS-2006}, over which ours has the benefit of reducing the complexity of the set of conditions from $\Sigma^0_3$ to $\Pi^0_2$.

\begin{definition}
\
\begin{enumerate}
\item A \emph{(computable Mathias) pre-condition} is a pair $(D,E)$ where $D$ is a finite set, $E$ is a computable set, and $\max D < \min E$.
\item A \emph{(computable Mathias) condition} is a pre-condition $(D,E)$, such that $E$ is infinite.
\item A pre-condition $(D',E')$ \emph{extends} a pre-condition $(D,E)$, written $(D',E') \leq (D,E)$, if $D \subseteq D' \subseteq D \cup E$ and $E' \subseteq E$.
\item A set $A$ \emph{satisfies} a pre-condition $(D,E)$ if $D \subseteq A \subseteq D \cup E$.
\end{enumerate}
\end{definition}

By an \emph{index} for a pre-condition $(D,E)$ we shall mean a pair $(d,e)$ such that $d$ is the canonical index of $D$ and $E = \{x : \Phi_e(x) \downarrow = 1\}$. In particular, if $E$ is finite, $\Phi_e$ need not be total. This definition makes the set of all indices $\Pi^0_1$, but we can pass to a computable subset containing an index for every pre-condition. Namely, define a strictly increasing computable function $g$ by
\[
\Phi_{g(d,e)}(x) =
\begin{cases}
0 & \text{if } x \leq \max D_d,\\
\Phi_e(x) & \text{otherwise.}
\end{cases}
\]
Then the set of pairs of the form $(d,g(d,e))$ is computable, and each is an index for a pre-condition. Moreover, if $(d,e)$ is an index as well, then it and $(d,g(d,e))$ index the same pre-condition. Formally, all references to pre-conditions in the sequel will be to indices from this set, and we shall treat $D$ and $E$ as numbers when convenient.

Note that whether one pre-condition extends another is a $\Pi^0_2$ question. By adopting the convention that for all $e$ and $x$, if $\Phi_e(x) \downarrow$ then $\Phi_e(y) \downarrow \in \{0,1\}$ for all $y \leq x$, the same question for conditions becomes $\Pi^0_1$.

In what follows, a \emph{$\Sigma^0_n$ set of conditions} refers to a $\Sigma^0_n$-definable set of pre-conditions, each of which is a condition. (Note that this is not the same as the set of all conditions satisfying a given $\Sigma^0_n$ definition, as discussed further in the next section.) As usual, we call such a set \emph{dense} if it contains an extension of every condition.

\begin{definition}
Fix $n \in \omega$.
\begin{enumerate}
\item A set $A$ \emph{meets} a set $\mathcal{C}$ of conditions if it satisfies some member of $\mathcal{C}$.
\item A set $A$ \emph{avoids} a set $\mathcal{C}$ of conditions if it meets the set of conditions having no extension in $\mathcal{C}$.
\item A set $G$ is \emph{Mathias $n$-generic} if it meets or avoids every $\Sigma^0_n$ set of conditions.
\item A set $G$ is \emph{weakly Mathias $n$-generic} if it meets every dense $\Sigma^0_n$ set of conditions.
\end{enumerate}
\end{definition}

\noindent We call a degree \emph{generic} if it contains a set that is $n$-generic for all $n$.

It is easy to see that for every $n \geq 2$, there exists a Mathias $n$-generic $G \leq_T \emptyset^{(n)}$ (indeed, even $G' \leq_T\emptyset^{(n)}$). This is done just as in Cohen forcing, but as there is no computable listing of $\Sigma^0_n$ sets of conditions, one goes through the $\Sigma^0_n$ sets of pre-conditions and checks which of these consist of conditions alone. We pass to some other basic properties of generics. We shall refer to Mathias $n$-generics below simply as $n$-generics when no confusion is possible.

\section{Basic results}\label{sec_basic}

Note that the set of all conditions is $\Pi^0_2$. Thus, the set of conditions satisfying a given $\Sigma^0_n$ definition is $\Sigma^0_n$ if $n \geq 3$, and $\Sigma^0_3$ otherwise. For $n < 3$, we may thus wish to consider the following stronger form of genericity, which has no analogue in the case of Cohen forcing.

\begin{definition}
A set $G$ is \emph{strongly $n$-generic} if, for every $\Sigma^0_n$-definable set of pre-conditions $\mathcal{C}$, either $G$ meets some condition in $\mathcal{C}$ or $G$ meets the set of conditions not extended by any condition in $\mathcal{C}$.
\end{definition}

\begin{proposition}\label{prop_n3}
For $n \geq 3$, a set is strongly $n$-generic if and only if it is $n$-generic. For $n \leq 2$, a set is strongly $n$-generic if and only if it is $3$-generic.
\end{proposition}

\begin{proof}
Evidently, every strongly $n$-generic set is $n$-generic. Now suppose $\mathcal{C}$ is a $\Sigma^0_n$ set of pre-conditions, and let $\mathcal{C}'$ consist of all the conditions in $\mathcal{C}$. An infinite set meets or avoids $\mathcal{C}$ if and only if it meets or avoids $\mathcal{C}'$, so every $\max\{n,3\}$-generic set meets or avoids $\mathcal{C}$. For $n \geq 3$, this means that every $n$-generic set is strongly $n$-generic, and for $n \leq 2$ that every $3$-generic set is strongly $n$-generic.

It remains to show that every strongly $0$-generic set is $3$-generic. Let $\mathcal{C}$ be a given $\Sigma^0_3$ set of conditions, and let $R$ be a computable relation such that $(D,E)$ belongs to $\mathcal{C}$ if and only if $(\exists a)(\forall x)(\exists y)R(D,E,a,x,y)$. Define a strictly increasing computable function $g$ by
\[
\Phi_{g(D,E,a)}(x) =
\begin{cases}
\Phi_e(x) & \text{if } (\exists y) R(D,E,a,x,y) \text{ and } \Phi_E(x) \downarrow,\\
\uparrow & \text{otherwise},
\end{cases}
\]
and let $\mathcal{C}'$ be the computable set of all pre-conditions of the form $(D,g(D,E,a))$. If $(D,E) \in \mathcal{C}$ then $\Phi_E$ is total and so there is an $a$ such that $\Phi_{g(D,E,a)} = \Phi_E$. If, on the other hand, $(D,E)$ is a pre-condition not in $\mathcal{C}$ then for each $a$ there is an $x$ such that $\Phi_{g(D,E,a)}(x) \uparrow$. Thus, the members of $\mathcal{C}$ are precisely the conditions in $\mathcal{C}'$, so an infinite set meets or avoids one if and only if it meets or avoids the other. In particular, every strongly $0$-generic set meets or avoids $\mathcal{C}$.
\end{proof}

As a consequence, we shall restrict ourselves to $3$-genericity or higher from now on. Without further qualification, $n$ below will always be a number $\geq 3$.

\begin{proposition}
Every $n$-generic set is weakly $n$-generic, and every weakly $n$-generic set is $(n-1)$-generic.
\end{proposition}

\begin{proof}
The first implication is clear. For the second, let a $\Sigma^0_{n-1}$ set $\mathcal{C}$ of conditions be given. Let $\mathcal{D}$ be the class of all conditions that are either in $\mathcal{C}$ or else have no extension in $\mathcal{C}$, which is clearly dense. If $n \geq 4$, then $\mathcal{D}$ is easily seen to be $\Sigma^0_n$ (actually $\Pi^0_{n-1}$) since saying a condition $(D,E)$ has no extension in $\mathcal{C}$ is expressed as
\[
\forall (D',E')[[(D',E') \text{ is a condition } \wedge (D',E') \leq (D,E)] \implies (D',E') \notin \mathcal{C}].
\]
If $n = 3$, this makes $\mathcal{D}$ appear to be $\Sigma^0_4$ but since $\mathcal{C}$ is a set of conditions only, we can re-write the above line as
\[
\forall (D',E')[(\forall x)[\Phi_{E'}(x) \downarrow = 1 \wedge \Phi_E(x) \downarrow \implies \Phi_E(x) = 1] \implies (D',E') \notin \mathcal{C}],
\]
which gives a $\Sigma^0_3$ definition. In either case, then, a weakly $n$-generic set must meet $\mathcal{D}$, and hence must either meet or avoid $\mathcal{C}$.
\end{proof}

The proof of the following proposition is straightforward.

\begin{proposition}\label{prop_genfacts}
Every weakly $n$-generic set $G$ is hyperimmune relative to $\emptyset^{(n-1)}$. If $G$ is $n$-generic, then its degree forms a minimal pair with $\0^{(n-1)}$.
\end{proposition}

\begin{corollary}
Not every $n$-generic set is weakly $(n+1)$-generic.
\end{corollary}

\begin{proof}
Take any $n$-generic $G \leq_T \emptyset^{(n)}$. Then $G$ is not hyperimmune relative to $\emptyset^{(n+1)}$, and so cannot be weakly $(n+1)$-generic.
\end{proof}

We shall separate weakly $n$-generic sets from $n$-generic sets in Section \ref{sec_degree}, thereby obtaining a strictly increasing sequence of genericity notions
\[
\text{weakly $3$-generic $\from$ $3$-generic $\from$ weakly $4$-generic $\from$ $\cdots$} 
\]
as in the case of Cohen forcing. In many other respects, however, the two types of genericity are very different. For instance, as noted in \cite[Section 4.1]{CJS-2001}, every Mathias generic $G$ is cohesive, i.e., satisfies $G \subseteq^* W$ or $G \subseteq^* \overline{W}$ for every computably enumerable set $W$. In particular, if we write $G = G_0 \oplus G_1$ then one of $G_0$ or $G_1$ is finite. This is be false for Cohen generics, which, by an analogue of van Lambalgen's theorem due to Yu \cite[Proposition 2.2]{Yu-2006}, have relatively $n$-generic halves. Thus, no Mathias generic can be even Cohen $1$-generic.

\begin{question}
What form of van Lambalgen's theorem holds for Mathias forcing?
\end{question}

Another basic fact is that every Mathias $n$-generic $G$ is high, i.e., satisfies $G' \geq_T \emptyset''$. (See \cite{BKLS-2006}, Corollary 6.7, or \cite{CJS-2001}, Section 5.1 for a proof.) By contrast, it is a well-known result of Jockusch \cite[Lemma 2.6]{Jockusch-1980} that every Cohen $n$-generic set $G$ satisfies $G^{(n)} \equiv_T G \oplus \emptyset^{(n)}$. As no high $G$ can satisfy $G'' \leq_T G \oplus \emptyset''$, it follows that no Mathias generic can have even Cohen 2-generic degree. The same argument does not prevent a Mathias $n$-generic from having Cohen $1$-generic degree, as there are high 1-generic sets, but we prove in Corollary \ref{cor_Mathno1Coh} that this does not happen either.

\section{The forcing relation}\label{sec_forcing}

Much of the discrepancy between Mathias and Cohen genericity stems from the fact that the complexity of forcing a formula, defined below, does not agree with the complexity of the formula.

We regard every $\Sigma^0_0$ formula $\varphi$ as being written in disjunctive normal form according to some fixed effective procedure for doing so. Let $n_\varphi$ denote the number of disjuncts. For each $i < n_\varphi$, let $P_{\varphi,i}$ be the set of all $n$ such that $n \in X$ is a conjunct of the $i$th disjunct, and let $N_{\varphi,i}$ be the set of all $n$ such that $n \notin X$ is a conjunct of the $i$th disjunct. Canonical indices for these sets can be found uniformly effectively from an index for $\varphi$.

\begin{definition}
Let $(D,E)$ be a condition and let $\varphi(X)$ be a formula in exactly one free set variable. If $\varphi$ is $\Sigma^0_0$, say $(D,E)$ \emph{forces} $\varphi(G)$, written $(D,E) \forces \varphi(G)$, if for some $i < n_{\varphi}$, $P_{\varphi,i} \subseteq D$ and $N_{\varphi,i} \subseteq \overline{D \cup E}$. From here, extend the definition of $(D,E) \forces \varphi(G)$ to arbitrary $\varphi$ inductively according to the standard definition of strong forcing.
\end{definition}

\begin{remark}\label{rem_sigma0}
Note that if $\varphi$ is $\Sigma^0_0$ and $A$ is any set then $\varphi(A)$ holds if and only if there is an $i < n_\varphi$ such that $P_{\varphi,i} \subseteq A$ and $N_{\varphi,i} \subseteq \overline{A}$. Hence, $(D,E) \forces \varphi(G)$ if and only if $\varphi(D \cup F)$ holds for all finite $F \subset E$.
\end{remark}

\begin{lemma}\label{lem_forcebounds}
Let $(D,E)$ be a condition and let $\varphi(X)$ be a formula in exactly one free set variable.
\begin{enumerate}
\item If $\varphi$ is $\Sigma^0_0$ then the relation $(D,E) \forces \varphi(G)$ is computable.
\item If $\varphi$ is $\Pi^0_1$, $\Sigma^0_1$, or $\Sigma^0_2$, then so is the relation $(D,E) \forces \varphi(G)$.
\item For $n \geq 2$, if $\varphi$ is $\Pi^0_n$ then the relation of $(D,E) \forces \varphi(G)$ is $\Pi^0_{n+1}$.
\item For $n \geq 3$, if $\varphi$ is $\Sigma^0_n$ then the relation $(D,E) \forces \varphi(G)$ is $\Sigma^0_{n+1}$.
\end{enumerate}
\end{lemma}

\begin{proof}
We first prove 1. If $\varphi$ is $\Sigma^0_0$ and $\varphi(D \cup F)$ does not hold for some finite $F \subset E$, then neither does
\[
\varphi(D \cup (F \cap (\bigcup_{i < n_\varphi} P_{\varphi,i} \cup N_{\varphi,i}))).
\]
So by Remark \ref{rem_sigma0}, we have that $(D,E) \forces \varphi(G)$ if and only if $\varphi(D \cup F)$ holds for all finite $F \subset E \cap (\bigcup_{i < n_\varphi} P_{\varphi,i} \cup N_{\varphi,i})$, which can be checked computably.

For 2, suppose that $\varphi(X) \equiv (\forall x)\theta(x,X)$, where $\theta$ is $\Sigma^0_0$. We claim that $(D,E)$ forces $\varphi(G)$ if and only if $\theta(a,D \cup F)$ holds for all $a$ and all finite $F \subset E$, which makes the forcing relation $\Pi^0_1$. The right to left implication is clear. For the other, suppose there is an $a$ and a finite $F \subset E$ such that $\theta(a,D \cup F)$ does not hold. Writing $\theta_a(X)$ for the formula $\theta(a,X)$, let $D' = D \cup F$ and
\[
E' = \{x \in E: x > \max D \cup F \cup \bigcup_{i < n_{\theta_a}} P_{\theta_a,i} \cup N_{\theta_a,i}\},
\]
so that $(D',E')$ is a condition extending $(D,E)$. Then if $(D'',E'')$ is any extension of $(D',E')$, we have that
\[
D'' \cap (\bigcup_{i < n_{\theta_a}} P_{\theta_a,i} \cup N_{\theta_a,i})) = (D \cup F) \cap (\bigcup_{i < n_{\theta_a}} P_{\theta_a,i} \cup N_{\theta_a,i})),
\]
and so $\theta(a,D'')$ cannot force $\theta(a,G)$. Thus $(D,E)$ does not force $\varphi(G)$. The rest of 2 follows immediately, since forcing a formula that is $\Sigma^0_1$ over another formula is $\Sigma^0_1$ over the complexity of forcing that formula.

We next prove 3 for $n = 2$. Suppose that $\varphi(G) \equiv (\forall x)(\exists y) \theta(x,y,X)$ where $\theta$ is $\Sigma^0_0$. Our claim is that $(D,E) \forces \varphi(G)$ if and only if, for every $a$ and every condition $(D',E')$ extending $(D,E)$, there is a finite $F \subset E'$ and a number $k > \max F$ such that
\begin{equation}\label{eq_1}
(D' \cup F, \{x \in E' : x > k\}) \forces (\exists y)\theta(a,y,G),
\end{equation}
which is a $\Pi^0_3$ definition. Since the condition on the left side of \eqref{eq_1} extends $(D',E')$, this definition clearly implies forcing. For the opposite direction, suppose $(D,E) \forces \varphi(G)$ and fix any $a$ and $(D',E') \leq (D,E)$. Then by definition, there is a $b$ and a condition $(D'',E'')$ extending $(D',E')$ that forces $\theta(a,b,G)$. Write $\theta_{a,b}(X) = \theta(a,b,X)$, and let $F \subset E'$ be such that $D'' = D' \cup F$. Since $\theta_{a,b}(D' \cup F)$ holds, we must have $P_{\theta_{a,b},i} \subseteq D' \cup F$ and $N_{\theta_{a,b},i} \cap (D' \cup F) = \emptyset$ for some $i < n_{\theta_{a,b}}$. Thus, if we let $k = \max N_{\theta_{a,b},i}$, we obtain \eqref{eq_1}.

To complete the proof, we prove 3 and 4 for $n \geq 3$ by simultaneous induction on $n$. Clearly, 3 for $n-1$ implies 4 for $n$, so we already have 4 for $n = 3$. Now assume 4 for some $n \geq 3$. The definition of forcing a $\Pi^0_{n+1}$ statement is easily seen to be $\Pi^0_2$ over the relation of forcing a $\Sigma^0_n$ statement, and hence $\Pi^0_{n+2}$ by hypothesis. Thus, 3 holds for $n+1$.
\end{proof}

We shall see in Corollary \ref{cor_strict} that the complexity bounds in parts 3 and 4 of the lemma cannot be lowered even to $\Delta^0_{n+1}$. One consequence is that an $n$-generic set only decides every $\Sigma^0_{n-1}$ formula, not necessarily every $\Sigma^0_n$ formula.

\begin{proposition}\label{prop_gentruth}
Let $G$ be $n$-generic, and for $m \leq n$ let $\varphi(X)$ be a $\Sigma^0_m$ or $\Pi^0_m$ formula in exactly one free set variable. If $(D,E)$ is any condition satisfied by $G$ that forces $\varphi(G)$, then $\varphi(G)$ holds.
\end{proposition}

\begin{proof}
If $m = 0$, then $\varphi$ holds of any set satisfying $(D,E)$, whether it is generic or not. If $m > 0$ and the result holds for $\Pi^0_{m-1}$ formulas, it also clearly holds for $\Sigma^0_m$ formulas. Thus, we only need to show that if $m > 0$ and the result holds for $\Sigma^0_{m-1}$ formulas then it also holds for $\Pi^0_m$ formulas. To this end, suppose $\varphi(X) \equiv (\forall x)\theta(x,X)$, where $\theta$ is $\Sigma^0_{m-1}$. For each $a$, let $\mathcal{C}_a$ be the set of all conditions forcing $\theta(a,X)$, which has complexity at most $\Sigma^0_n$ by Lemma~\ref{lem_forcebounds}. Hence, $G$ meets or avoids each $\mathcal{C}_a$. But if $G$ were to avoid some $\mathcal{C}_a$, say via a condition $(D',E')$, then $(D',E')$ would force $\neg \theta(a,G)$, and then $(D,E)$ and $(D',E')$ would have a common extension forcing $\theta(a,G)$ and $\neg \theta(a,G)$. Thus, $G$ meets every $\mathcal{C}_a$, so $\theta(a,G)$ holds for all $a$ by hypothesis, meaning $\varphi(G)$ holds.
\end{proof}

\begin{remark}\label{rem_formneg}
It is not difficult to see that if $\varphi(G)$ is the negation of a $\Sigma^0_m$ formula then any condition $(D,E)$ forcing $\varphi(G)$ forces an equivalent $\Pi^0_m$ formula. Thus, if $G$ is $n$-generic and satisfies such a condition, then $\varphi(G)$ holds.
\end{remark}

\section{Degrees of Mathias generics}\label{sec_degree}

We begin here with a jump property for Mathias generics similar to that of Jockusch for Cohen generics. It follows that the degrees $\mathbf{d}$ satisfying $\mathbf{d}^{(n-1)} = \mathbf{d'} \cup \mathbf{0}^{(n-1)}$ yield a strict hierarchy of subclasses of the high degrees.

\begin{theorem}\label{thm_jumpgen}
For all $n \geq 2$, if $G$ is $n$-generic then $G^{(n-1)} \equiv_T G' \oplus \emptyset^{(n)}$.
\end{theorem}

\begin{proof}
That $G^{(n-1)} \geq_T G' \oplus \emptyset^{(n)}$ follows from the fact that $G' \geq_T \emptyset''$. To show $G^{(n-1)} \leq_T G' \oplus \emptyset^{(n)}$, we wish to decide every $\Sigma^{0,G}_{n-1}$ sentences using $G' \oplus \emptyset^{(n)}$. Let $\varphi_0(X),\varphi_1(X),\ldots$, be a computable enumeration of all $\Sigma^0_{n-1}$ sentences in exactly one free set variable, and for each $i$ let $\mathcal{C}_i$ be the set of conditions forcing $\varphi_i(G)$, and $\mathcal{D}_i$ the set of conditions forcing $\neg \varphi(G)$. Then $\mathcal{D}_i$ is the set of conditions with no extension in $\mathcal{C}_i$, so if $G$ meets $\mathcal{C}_i$ it cannot also meet $\mathcal{D}_i$. On the other hand, if $G$ avoids $\mathcal{C}_i$ then it meets $\mathcal{D}_i$ by definition. Now by Lemma \ref{lem_forcebounds}, each $\mathcal{C}_i$ is $\Sigma^0_n$ since $n \geq 3$, and so it is met or avoided by $G$. Thus, for each $i$, either $G$ meets $\mathcal{C}_i$, in which case $\varphi(G)$ holds by Proposition \ref{prop_gentruth}, or else $G$ meets $\mathcal{D}_i$, in which case $\neg \varphi(G)$ holds by Remark \ref{rem_formneg}. To conclude the proof, we observe that $G' \oplus \emptyset^{(n)}$ can decide, uniformly in $i$, whether $G$ meets $\mathcal{C}_i$ or $\mathcal{D}_i$. Indeed, from a given $i$, indices for $\mathcal{C}_i$ and $\mathcal{D}_i$ (as a $\Sigma^0_n$ set and a $\Pi^0_n$ set, respectively) can be found uniformly computably, and then $\emptyset^{(n)}$ has only to produce these sets until a condition in one is found that is satisfied by $G$, which can in turn be determined by $G'$.
\end{proof}

\begin{corollary}\label{cor_strict}
For every $n \geq 2$ there is a $\Pi^0_n$ formula in exactly one free set variable, the relation of forcing which is not $\Delta^0_{n+1}$. For $n \geq 3$, there is also a $\Sigma^0_n$ such formula.
\end{corollary}

\begin{proof}
By the proof of Lemma \ref{lem_forcebounds}, the second part implies the first, so it suffices to prove it. If forcing every $\Sigma^0_n$ formula $\varphi(X)$ were $\Delta^0_{n+1}$, then the proof of Theorem \ref{thm_jumpgen} could be carried out computably in $G' \oplus \emptyset^{(n-1)}$ instead of $G' \oplus \emptyset^{(n)}$. Hence, we would have $G^{(n-1)} \equiv_T G' \oplus \emptyset^{(n-1)}$. But as $G$ is high, $\emptyset^{(n-1)} \leq_T G^{(n-2)}$, so this would yield $G^{(n-1)} \leq_T G^{(n-2)}$, a contradiction.
\end{proof}

The following result is the analogue of Theorem 2.3 of Kurtz \cite{Kurtz-1983} that every $A >_T \emptyset^{(n-1)}$ hyperimmune relative to $\emptyset^{(n-1)}$ is Turing equivalent to the $(n-1)$st jump of a weakly Cohen $n$-generic set. The proof, although mostly similar, requires a few important modifications. The main problem is in coding $B$ into $A^{(n-2)}$, which, in the case of Cohen forcing, is done by appending long blocks of $1$s to the strings under construction. As the infinite part of a Mathias condition can be made very sparse, we cannot use the same idea here. We highlight the changes below, and sketch the rest of the details.

\begin{proposition}
If $A >_T \emptyset^{(n-1)}$ is hyperimmune relative to $\emptyset^{(n-1)}$, then $A \equiv_T G^{(n-2)}$ for some weakly $n$-generic set $G$.
\end{proposition}

\begin{proof}
Computably in $A$, we build a sequence of conditions $(D_0,E_0) \geq (D_1,E_1) \geq \cdots$ beginning with $(D_0,E_0) = (\emptyset,\omega)$. Let $\mathcal{C}_0,\mathcal{C}_1,\ldots$ be a listing of all $\Sigma^0_n$ sets of pre-conditions, and fixing a $\emptyset^{(n-1)}$-computable enumeration of each $\mathcal{C}_i$, let $\mathcal{C}_{i,s}$ be the set of all pre-conditions enumerated into $\mathcal{C}_i$ by stage $p_A(s)$. We may assume that $\pair{D,E} \leq s$ for all $(D,E) \in \mathcal{C}_{i,s}$. Let $B_0,B_1,\ldots$ be a uniformly $\emptyset^{(n-1)}$-computable sequence of pairwise disjoint co-immune sets. Say $\mathcal{C}_i$ \emph{requires attention} at stage $s$ if there exists $b \leq p_A(s)$ in $B_i \cap E_s$ and a condition $(D,E)$ in $\mathcal{C}_{i,s}$ extending $(D_s \cup \{b\},\{x \in E_s: x > b\})$.

At stage $s$, assume $(D_s,E_s)$ is given. If there is no $i \leq s$ such that $\mathcal{C}_i$ requires attention at stage $s$, set $(D_{s+1},E_{s+1}) = (D_s,E_s)$. Otherwise, fix the least such $i$. Choose the least corresponding $b$ and earliest enumerated extension $(D,E)$ in $\mathcal{C}_{i,s}$, and let $(D',E') = (D,E)$. Then obtain $(D'',E'')$ from $(D',E')$ by forcing the jump, in the usual manner. Finally, let $k$ be the number of stages $t < s$ such that $(D_t,E_t) \neq (D_{t+1},E_{t+1})$, and let $(D''',E''') = (D'' \cup \{b\},\{x \in E'': x > b\})$, where $b$ is the least element of $B_{A(k)} \cap E''$. If $\pair{D''',E'''} \leq s+1$, set $(D_{s+1},E_{s+1}) = (D''',E''')$, and otherwise set $(D_{s+1},E_{s+1}) = (D_s,E_s)$.

By definition, the $B_i$ must intersect every computable set infinitely often, and so the entire construction is $A$-computable. That $G = \bigcup_s D_s$ is weakly $n$-generic can be verified much like in Kurtz's proof, but using the $\emptyset^{(n-1)}$-computable function $h$ where $h(s)$ is the least $t$ so that for each $(D,E)$ with $\pair{D,E} \leq s$ there exists $b \leq t$ in $B_i \cap E$ and $(D',E') \in C_{i,t}$ extending $(D \cup \{b\},\{x \in E: x > b\})$. That $G^{(n-2)} \leq_T A$ follows by Theorem \ref{thm_jumpgen} from $G'$ being forced during the construction and thus being $A$-computable. Finally, to show $A \leq_T G^{(n-2)}$, let $s_0 < s_1 < \cdots$ be all the stages $s > 0$ such that $(D_{s-1},E_{s-1}) \neq (D_{s},E_{s})$. The sequence $(D_{s_0},E_{s_0}) > (D_{s_1},E_{s_1}) \cdots$ can be computed by $G^{(n-2)}$ as follows. Given $(D_{s_k},E_{s_k})$, the least $b \in G - D_{s_k}$ must belong to some $B_i$, and since $G^{(n-2)}$ computes $\emptyset^{(n-1)}$ it can tell which $B_i$. Then $G^{(n-2)}$ can produce $\mathcal{C}_i$ until the first $(D',E')$ extending $(D_{s_k} \cup \{b\},\{x \in E_{s_k} : x > b\})$, and then obtain $(D'',E'')$ from $(D',E')$ by forcing the jump. By construction, $G$ satisfies $(D'',E'')$ and $(D_{s_{k+1}},E_{s_{k+1}}) = (D'' \cup \{b\},\{x \in E'' : x > b\})$ for the least $b \in G - D_{s_{k+1}}$. And this $b$ is in $B_1$ or $B_0$ depending as $k$ is or is not in $B$.
\end{proof}

\begin{corollary}
Not every weakly $n$-generic set is $(n+1)$-generic.
\end{corollary}

\begin{proof}
By the previous proposition, $\emptyset^{(n)} \equiv_T G^{(n-2)}$ for some weakly $n$-generic set $G$. By Theorem \ref{thm_jumpgen}, if $G$ were $n$-generic we would have $\emptyset^{(n+1)} \equiv_T G^{(n-1)} \equiv_T G' \oplus \emptyset^{(n)} \equiv_T \emptyset^{(n)}$, which cannot be.
\end{proof}

In spite of Theorem \ref{thm_jumpgen}, we are still left with the possibility that some Mathias $n$-generic set has Cohen $1$-generic degree. We now show that this cannot happen.

\begin{theorem}
If $G$ is $n$-generic then it has $\mathbf{GH}_1$ degree, i.e., $G' \equiv_T (G \oplus \emptyset')'$.
\end{theorem}

\begin{proof}
A condition $(D,E)$ forces $i \in (G \oplus \emptyset')'$ if there is a $\sigma \in 2^{<\omega}$ such that that $\Phi^\sigma_i(i) \downarrow$ and for all $x < |\sigma|$,
\[
\sigma(x) = 1 \implies (D,E) \forces x \in G \oplus \emptyset' \text{ and } \sigma(x) = 0 \implies (D,E) \forces x \notin G \oplus \emptyset'.
\]
This is thus a $\Sigma^0_2$ relation, as forcing $x \in G \oplus \emptyset'$ and $x \notin G \oplus \emptyset'$ are $\Sigma^0_1$ and $\Pi^0_1$, respectively. We claim that $(D,E)$ forcing $i \notin (G \oplus \emptyset')'$, i.e., $\neg (i \in (G \oplus \emptyset')')$, is equivalent to $(D,E)$ having no finite extension that forces $i \in (G \oplus \emptyset')'$, and hence is $\Pi^0_2$. That forcing implies this fact is clear. In the other direction, suppose $(D,E)$ does not force $i \notin (G \oplus \emptyset')'$, and so has an extension $(D',E')$ that forces $i \notin (G \oplus \emptyset')'$. Let $\sigma$ witness this fact, as above. Then if $P$ and $N$ consist of the $x < |\sigma|$ such that $\sigma(2x) = 1$ and $\sigma(2x) = 0$, respectively, $\sigma$ witnesses that $(D \cup P, \{ x \in E: x > \max P \cup N \})$ also forces $i \in (G \oplus \emptyset')'$.

We now show that $G' \geq_T (G \oplus \emptyset')'$. Let $\mathcal{C}_i$ be the set of conditions that force $i \in (G \oplus \emptyset')'$, and $\mathcal{D}_i$ the set of conditions that force $i \notin (G \oplus \emptyset')'$. Then $\mathcal{C}_i$ is $\Sigma^0_3$ and $\mathcal{D}_i$ is $\Pi^0_2$, and indices for them as such can be found uniformly from $i$. Each $\mathcal{C}_i$ must be either met or avoided by $G$, and as in the proof of Theorem \ref{thm_jumpgen}, $G$ meets $\mathcal{C}_i$ if and only if it does not meet $\mathcal{D}_i$. Which of the two is the case can be determined by $G'$ since $G' \geq_T \emptyset''$ and $\mathcal{C}_i$ and $\mathcal{D}_i$ are both c.e.\ in $\emptyset''$. By Proposition \ref{prop_gentruth}, $G'$ can thus determine whether $i \in (G \oplus \emptyset')'$, as desired.
\end{proof}

Recall that a degree $\mathbf{d}$ is $\mathbf{GL}_n$ if $\mathbf{d}^{(n)} = \mathbf{d} \cup \0^{(n)}$, and that no such degree can be $\mathbf{GH}_1$. It was shown by Jockusch and Posner \cite[Corollary 7]{JP-1978} that every $\overline{\mathbf{GL}}_2$ degree computes a Cohen 1-generic set. Hence, we obtain the following:

\begin{corollary}\label{cor_Mathno1Coh}
Every Mathias $n$-generic set has $\overline{\mathbf{GL}}_m$ degree for all $m \geq 1$. Hence, it is not of Cohen 1-generic degree, but it does compute a Cohen 1-generic.
\end{corollary}

The corollary leaves open the following question, which we have so far been unable to answer. The subsequent results give partial answers.

\begin{question}
Does every Mathias $n$-generic set compute a Cohen $n$-generic set?
\end{question}

\begin{theorem}
If $G$ is Mathias $n$-generic, and $A \leq_T \emptyset^{(n-1)}$ is bi-immune, then $G \oplus A$ computes a Cohen $n$-generic.
\end{theorem}

\begin{proof}
Let $\mathcal{C}_0,\mathcal{C}_1,\ldots$ be a listing of all $\Sigma^0_n$ subsets of $2^{<\omega}$, together with fixed $\emptyset^{(n-1)}$-computable enumerations. For each $i$, let $\mathcal{D}_i$ be the set of all conditions $(D,E)$ such that $D \cap A$, viewed as a binary string of length $\min E$, belongs to $\mathcal{C}_i$. Then $\mathcal{D}_i$ is a $\Sigma^0_n$ set of conditions, and as such must be met or avoided by $G$. If $G$ meets $\mathcal{D}_i$ then $G \cap A$, viewed as an element of $2^\omega$, meets $\mathcal{C}_i$. If $G$ avoids $\mathcal{D}_i$, we claim that $G \cap A$ must avoid $\mathcal{C}_i$. Indeed, suppose $G$ avoids $\mathcal{D}_i$ via $(D,E)$. Since $A$ and $\overline{A}$ are each co-immune, they intersect $E$ infinitely often, and so if $D \cap B$ had an extension $\tau$ in $\mathcal{C}_i$, we could make a finite extension $(D',E')$ of $(D,E)$ so that $D' \cap A = \tau$. This extension would belong to $\mathcal{D}_i$, a contradiction.
\end{proof}

It follows, for example, that the join of $G$ with any non-computable $S \leq_T \emptyset'$ computes a Cohen $n$-generic. Our last result shows that the kind of coding employed in the above theorem cannot be used also for $S = \emptyset$.

\begin{proposition}
If $G$ is Mathias $n$-generic and $H$ is Cohen $n$-generic then $H$ is not many-one reducible to $G$.
\end{proposition}

\begin{proof}
Seeking a contradiction, suppose $f$ is a computable function such that $f(H) \subseteq G$ and $f(\overline{H}) \subseteq \overline{G}$. Since $G$ is cohesive and $\ran(f)$ is c.e., and since $H$ is non-computable, it follows that $G \subseteq^* \ran(f)$. Thus, for all sufficiently large $a$,
\[
a \in G \Iff (\forall x)[f(x) = a \implies x \in H] \Iff (\exists x)[f(x) = a \wedge x \in H]. 
\]
We conclude that $G \leq_T H$, and hence that $G \equiv_T H$. But this contradicts our observation at the end of Section \ref{sec_basic} that no Mathias $n$-generic can have Cohen $n$-generic degree.
\end{proof}

The following question is inspired by Proposition 2.8 of Jockusch \cite{Jockusch-1980}.

\begin{question}
If $\mathbf{a}$ and $\mathbf{b}$ are two Mathias generic degrees, must $\mathcal{D}(\mathbf{\leq a})$ and $\mathcal{D}(\mathbf{\leq b})$ be elementarily equivalent?
\end{question}

\end{document}